\documentclass[11pt]{amsart}
\usepackage{graphicx}
\usepackage[active]{srcltx}
 \makeatletter
\renewcommand*\subjclass[2][2000]{%
  \def\@subjclass{#2}%
  \@ifundefined{subjclassname@#1}{%
    \ClassWarning{\@classname}{Unknown edition (#1) of Mathematics
      Subject Classification; using '1991'.}%
  }{%
    \@xp\let\@xp\subjclassname\csname subjclassname@#1\endcsname
  }%
}
 \makeatother
\usepackage{enumerate,amssymb,  mathrsfs,yhmath}

\newtheorem{theorem}{Theorem}[section]
\newtheorem{lemma}[theorem]{Lemma}
\newtheorem*{lemma*}{Lemma}

\def\1ton{1,2,\ldots,n}

\usepackage{amssymb}
\usepackage{amsthm}
\usepackage{mathrsfs, amsfonts, amsmath}
\usepackage{graphicx}

\theoremstyle{definition}

\theoremstyle{remark}
\newtheorem{remark}[theorem]{Remark}

\numberwithin{equation}{section}

\def\XXint#1#2#3{{\setbox0=\hbox{$#1{#2#3}{\int}$}
\vcenter{\hbox{$#2#3$}}\kern-.5\wd0}}

\def\ge{\geqslant}
\begin{document}
\title{Sharp pointwise estimate of $\alpha-$harmonic functions}  \subjclass[2020]{Primary 30H10 }


\author{David Kalaj}
\address{University of Montenegro, Faculty of Natural Sciences and
Mathematics, Cetinjski put b.b. 81000 Podgorica, Montenegro}
\email{davidk@ucg.ac.me}


\begin{abstract}
Let $\alpha>-1$ and assume that $f$ is $\alpha-$harmonic mapping defined in the unit disk that belongs to the Hardy class $h^p$ with $p\ge 1$.
We obtain some sharp estimates of the type $|f(z)|\le g(|r|) \|f^\ast\|_p$ and $|Df(z)|\le h(|r|)\|f^\ast\|_p$. We also prove  a Schwarz type lemma for the class of  $\alpha-$harmonic mappings of the unit disk onto itself fixing the origin.
\end{abstract}
\maketitle
\tableofcontents
\section{Introduction and statement of main results}
In this paper $\mathbb{D}$ is the unit disk and $\mathbb{T}=\partial\mathbb{D}$ is its boundary.
For $\alpha>-1$ and $z\in \mathbb{D}$  let
$$T[f] =-\frac{\alpha^2}{4}(1-|z|^2)^{-\alpha-1} +\frac{\alpha}{2}(1-|z|^2)^{-\alpha-1}\left(z\frac{\partial}{\partial z}+\bar z\frac{\partial }{\partial \bar z}\right)+(1-|z|^2)^{-\alpha}\frac{\partial^2}{\partial z\partial \bar z}$$
be the second order elliptic partial differential operator.
Of particular interest to our analysis is the homogeneous partial differential equation
\begin{equation}\label{interest} T[f] = 0 \text{ in } \ \ \ \mathbb{ D},\end{equation}  and its associated the Dirichlet boundary value problem as follows

\begin{equation}\label{interest1}\begin{split} T[f] &= 0 \text{ in } \ \ \ \mathbb{ D}\\ f|_{\mathbb{T}}&=f^\ast .\end{split}\end{equation}

Here the boundary data $f^\ast\in \mathcal{D}'(\mathbb{T})$ is a distribution on $\mathbb{T}$, and the boundary condition on \eqref{interest1} is in the distribution sense $f_r\rightarrow f^\ast$ in   $\mathcal{D}'(\mathbb{T})$ as $r\rightarrow 1$, where $f_r(e^{it})=f(re^{it})$, $e^{it}\in \mathbb{T}$.
Then $f$ is a  $\mathscr{C}^2$ solution to the equation \eqref{interest} with the boundary condition $f|_{\mathbb{T}}=f^\ast$, if and only if
 \begin{equation}\label{poisalpa}f(z)=P_\alpha[F]=\frac{1}{2\pi}\int_0^{2\pi}K_\alpha(z e^{-it})f^\ast(e^{it})dt \end{equation} where $$K_\alpha(z) = c_\alpha\frac{(1-|z|^2)^{\alpha+1}}{|1-z|^{\alpha+2}},$$ $$c_\alpha=\frac{\Gamma(1+\alpha/2)^2}{\Gamma(1+\alpha)}.$$
The constant $c_\alpha$ is chosen in order to ensure that $$\lim_{r\to 1} \frac{1}{2\pi}\int_0^{2\pi}K_\alpha(z e^{-it})dt =1.$$  We refer to this class of mappings as the class of $\alpha-$harmonic mappings and recall that for the case $\alpha=0$, the class coincides with the class of usual harmonic mappings.
For the above result see \cite{olo}. For a more general context, we refer to the paper \cite{ahern}. We say that $f$ belongs to Hardy space $h^p$ if $f^\ast\in L^p(\mathbb{T})$, where $p> 1$ (\cite{ahern}).
The (normalized)  norm in $L^p(\mathbb{T})$ is defined by $$\|f\|_p=\left\{
                                                        \begin{array}{ll}
                                                          \left(\int_{0}^{2\pi}|f(e^{it})|^p\frac{dt}{2\pi}\right)^{1/p}, & \hbox{for $p<\infty$;} \\
                                                          \mathrm{ess\,sup}_{\zeta\in \mathbb{T}}|f(\zeta)|, & \hbox{for $p=\infty$.}
                                                        \end{array}
                                                      \right.$$
We prove some sharp results for the class of $\alpha-$harmonic functions and their derivatives. For similar results but for the Euclidean harmonic mappings in the unit disk see \cite{col, caot, kp}. For the multidimensional setting see the papers \cite{km1, bur, ABR, liu, aim, kh}. For hyperbolic harmonic or so-called $n-$harmonic mappings in the unit ball see \cite{chen1,chen2, bur, jmaa}. For some other related results for $\alpha-$harmonic mappings see \cite{chen}.

We prove the following theorems

\begin{theorem}\label{one} For $p\ge 1$,  there is a function $B_{\alpha,p}(r)$ and a constant $b_{\alpha,p}=\max_r B_{\alpha,p}(r)$ defined in \eqref{bealr} and \eqref{beal} below,  so that
for $f^\ast\in L^p(\mathbb{T})$ and for $z\in\mathbb{D}$ we have \begin{equation}\label{prima}
|f(z)|\le \frac{B_{\alpha,p}(r)}{(1-r^2)^{1/p}}\|f^\ast\|_p
\end{equation}
and
\begin{equation}\label{seconda}
|f(z)|\le \frac{b_{\alpha,p}}{(1-r^2)^{1/p}}\|f^\ast\|_p.
\end{equation}
The function $B$ and the constant $b$ are sharp. In particular for every $z\in \mathbb{D}$, $$|f(z)|< \|f\|_{L^\infty(\mathbb{T})}.$$
\end{theorem}
For a real or a complex-valued differentiable function $w=u+iv: \mathbb{D}\to\mathbb{C} $ we define the differential matrix by $$Dw(z)=\left(
                                                                                                                                      \begin{array}{cc}
                                                                                                                                        u_x & u_y \\
                                                                                                                                        v_x & v_y \\
                                                                                                                                      \end{array}
                                                                                                                                    \right).$$ here $z=x+i y$.  We define its norm by $$|Dw(z)| =\sup_{|h|=1} |Dw(z) h|.$$ It is well-known that $$|Dw(z)|=|w_z|+|w_{\bar z}|.$$
\begin{theorem}\label{due} For $p\ge 1$,  there is a function $C_{\alpha,p}(r)$ and a constant $c_{\alpha,p}=\max_r C_{\alpha,p}(r)$ (see \eqref{cap} and \eqref{Cap} below) which are asymptotically sharp as $\alpha\to 0$ so that
for $f^\ast\in L^p(\mathbb{T})$ and for $z\in\mathbb{D}$ we have  \begin{equation}\label{prima1}
|Df(z)|\le \frac{C_{\alpha,p}(r)}{\left(1-r^2\right)^{1+1/p}}\|f^\ast\|_p
\end{equation}
and
\begin{equation}\label{seconda1}
|Df(z)|\le \frac{c_{\alpha,p}}{\left(1-r^2\right)^{1+1/p}}\|f^\ast\|_p.
\end{equation}
Here $r=|z|$.
In particular, we have the sharp inequality

\begin{equation}\label{const}|Df(0)|\le\frac{(2+\alpha)\Gamma\left[1+\frac{\alpha}{2}\right]^2 }{ \Gamma[1+\alpha]} \left(\frac{\Gamma\left[\frac{1+q}{2}\right]}{\sqrt{\pi}\Gamma\left[\frac{2+q}{2}\right]}\right)^{\frac{1}{q}}\|f^\ast\|_p.\end{equation}
Here $q=p/(p-1)$.
\end{theorem}
To formulate and to prove our next results recall the basic definition of hypergeometric functions. For two positive integers $m$ and $n$ and vectors $a=(a_1,\dots, a_m)$ and $b=(b_1,\dots, b_n)$ we set $${_mF_n}\left[\begin{array}{c}
                                            a \\
                                            b
                                          \end{array};x\right] = \sum_{k=0}^\infty \frac{(a_1)_k\cdots (a_m)_k}{(b_1)_k \cdots (b_n)_k \cdot k!} x^k,$$ where $(y)_k: =\frac{\Gamma(y+k)}{\Gamma(y)}= y(y+1)\dots (y+k-1)$ is the Pochhammer symbol. The hypergeometric series converges at least for $|x|<1$. For basic properties and formulas concerning hypergeometric series, we refer to the book \cite{hyper}.
\begin{theorem}[{Schwarz lemma for $\alpha-$harmonic mappings}]\label{three}
Assume that $f: \mathbb{D}\to \mathbb{D}$ is a $\alpha$-harmonic function with $f(0)=0$. Then for $r=|z|$,  $$|f(z)|\le \frac{2 (2+\alpha) r \left(1-r^2\right)^{1+\alpha} \Gamma\left[1+\frac{\alpha}{2}\right]^2 F\left[\begin{array}{ccc}
                                            1 \ \ \ 1+\frac{\alpha}{4} \ \ \ \frac{3}{2}+\frac{\alpha}{4} \\
                                             \frac{3}{2} \ \ \  \frac{3}{2}
                                          \end{array}; \frac{4 r^2}{\left(1+r^2\right)^2}\right]}{\left(1+r^2\right)^{2+\frac{\alpha}{2}} \pi \Gamma[1+\alpha]}.$$
\end{theorem}
\begin{remark}
For $\alpha=0$ the previous inequality can be written as $$|f(z)|\le \frac{4}{\pi}\arctan |z|$$ and this inequality corresponds to the standard Schwarz lemma for harmonic functions.
\end{remark}
\section{Proofs}

\begin{proof}[Proof of Theorem~\ref{one}]

We start with the formula
$$f(z)=P_\alpha[f^\ast]=\int_0^{2\pi}K_\alpha(z e^{-it})f^\ast(e^{it})\frac{dt}{2\pi} .$$
Then by H\"older inequality we have

$$|f(z)|\le\frac{1}{2\pi}  \|f^\ast\|_p\left(\int_0^{2\pi} |K_\alpha (z e^{-it})|^q\frac{dt}{2\pi}\right)^{1/q}.$$ Let $z=r e^{ix}$. After the change of variables $$e^{i(t-x)}=\frac{r+e^{is}}{1+r e^{is}},$$ $$dt=\frac{1-r^2}{1+r^2+2 r \cos s}ds,$$
we get
$$|f(z)|\le c_\alpha \|f^\ast\|_p\left(\int_0^{2\pi} A(q,s)\frac{ds}{2\pi}\right)^{1/q}.$$

Where  $$A(q,s)=\left(1-r^2\right)^{1-q} \left(1+r^2+2 r \cos s\right)^{-1+q+\frac{\alpha q}{2}}.$$
Now $$B_{\alpha,p}(r)=c_\alpha\left( \int_0^{2\pi}\left(1+r^2+2 r \cos s\right)^{-1+q+\frac{\alpha q}{2}}\frac{dt}{2\pi}\right)^{1/q}.$$ Then for $m=q(1+\alpha/2)-1,$ by using the Taylor development $$\left(1+r^2+2 r \cos s\right)^m=\sum_{n=0}^\infty 2^n r^n \left(1+r^2\right)^{m-n} \binom{m}{n}\cos^n s$$ and the formula $$\int_0^{2\pi}\cos^n sds=\frac{\left(1+(-1)^n\right) \sqrt{\pi } \Gamma\left[\frac{1+n}{2}\right]}{\Gamma\left[\frac{2+n}{2}\right]}$$ we obtain
\begin{equation}\label{bealr}B_{\alpha,p}(r)=c_\alpha\left({\left(1+r^2\right)^m} F\left[\frac{1}{2}-\frac{m}{2},-\frac{m}{2},1;\frac{4 r^2}{\left(1+r^2\right)^2}\right]\right)^{1/q},\end{equation} where $F$ is the Gaussian hypergoemtric function and
\begin{equation}\label{beal}b_{\alpha,p}=B_{\alpha,p}(1)=c_\alpha\left(\frac{2^{-1+(2+\alpha) q}  \Gamma\left[-\frac{1}{2}+q+\frac{\alpha q}{2}\right]}{2\sqrt{\pi } \Gamma\left[q+\frac{\alpha q}{2}\right]}\right)^{1/q}.\end{equation}
The last fact follows from Lemma~\ref{bele} below.
Then we get $$|f(z)|< (1-|z|^2)^{1/q-1} B_{\alpha,p}(r)\|f^\ast\|_p,$$ and $$|f(z)|< (1-|z|^2)^{1/q-1} b_{\alpha,p}\|f^\ast\|_p.$$

\end{proof}
\begin{proof}[Proof of Theorem~\ref{due}]
We first have that $$Df(z) h=\frac{1}{2\pi}\int_0^{2\pi}\left<\nabla K_\alpha(z e^{-it}), h\right>f^\ast(e^{it})dt,$$ where $\nabla$ denotes the gradient w.r.t. $z$. Let $h=e^{i\tau}$. Let $q\ge 1$, $1/p+1/q=1$.  Since $$\nabla u = u_x+ i u_y = 2\bar\partial u,$$ we get
\begin{equation}\label{weget}|Df(z)|\le \|f^\ast\|_p \max_\tau\left(\int_0^{2\pi}|\Re\left[ 2\bar\partial \left(K_\alpha(z e^{-it})\right) e^{-i\tau}\right]|^q\frac{dt}{2\pi}\right)^{1/q} .\end{equation}

Here $$2\bar\partial \left(K_\alpha(z e^{-it})\right)=\frac{ c_\alpha(1-|z|^2)^\alpha \left(2 (1+\alpha) z-e^{i t} (2+\alpha+\alpha |z|^2)\right)}{|1-ze^{-it}|^{\alpha+2} \left( \bar z e^{it}-1\right) }$$

Then for $h=e^{i\tau}$, $z=r e^{is}$, $t=c+s$, then $$\bar\partial \left(K_\alpha(z e^{-it})\right)$$ can be written as
$$\frac{c_\alpha \left(1-r^2\right)^\alpha \left(-2 (1+\alpha) r+e^{i c} \left(2+\alpha+\alpha r^2\right)\right) \left(1+r^2-2 r \cos c\right)^{-\alpha/2}}{  e^{i (\tau-c-s)}\left(e^{i c}-r\right) \left(-1+e^{i c} r\right)^2}$$

After making the change in \eqref{weget}  $$e^{ic}=\frac{r+e^{ib}}{1+e^{ib} r}, \ \  dc=\frac{1-r^2}{1+r^2+2r\cos b} db$$ we arrive at the subintegral expression

$$SE=\frac{c^q_\alpha |\alpha r \cos(t-s)-(2+\alpha) \cos(b-t+s)|^q}{2\pi  \left(1-r^2\right)^{2 q-1} \left(1+r^2+2 r \cos b\right)^{1-q-\frac{\alpha q}{2}}}  ,$$ so we get

\begin{equation}\label{weget2}\begin{split}|Df(z)|&\le  \frac{c_\alpha \|f^\ast\|_p}{  \left(1-r^2\right)^{1+ 1/p}}   \max_\eta \left(\int_0^{2\pi}\frac{|\alpha r \cos\eta-(2+\alpha) \cos(b+\eta)|^q }{\left(1+r^2+2 r \cos b\right)^{1-q-\frac{\alpha q}{2}}}\frac{db}{2\pi}\right)^{1/q}.\end{split} \end{equation}
Now we prove the following lemma
\begin{lemma}\label{marte}
Let $q\ge 0$ and $r\in[0,1]$. Define
$$H(y):=\int_0^{2\pi} |\cos(x)|^q (1+r^2+2 r \cos (x-y))^m dx.$$ Then $$\max H(y)=\left\{
                                                                                   \begin{array}{ll}
                                                                                     H(\pi/2), & \hbox{for $m\le 1$;} \\
                                                                                     H(0), & \hbox{for $m\ge 1$.}
                                                                                   \end{array}
                                                                                 \right. $$

\end{lemma}

\begin{proof}[Proof of Lemma~\ref{marte}] We have
\[\begin{split}H'(y)&=2 m r \int_0^{2\pi}|\cos x|^p \left(1+r^2+2 r \cos (x-y)\right)^{m-1} \sin (x-y)dx
\\&=2 m r \int_0^{2\pi}|\cos (x+y)|^p \left(1+r^2+2 r \cos x\right)^{m-1} \sin x dx
\\& =2 m r \int_0^{\pi}|\cos (x+y)|^p \left(1+r^2+2 r \cos x\right)^{m-1} \sin x dx
\\& +2 m r \int_{\pi}^{2\pi}|\cos (x+y)|^p \left(1+r^2+2 r \cos x\right)^{m-1} \sin x dx
\\&=2 m r \int_0^{\pi}|\cos (x+y)|^p \left(1+r^2+2 r \cos x\right)^{m-1} \sin x dx
 \\&-2 m r \int_{0}^{\pi}|\cos (x+y)|^p \left(1+r^2-2 r \cos x\right)^{m-1} \sin x dx.
\end{split}\]
Now let $$S(x,y)=\left(1+r^2-2 r \sin x\right)^{m-1} -\left(1+r^2+2 r \sin x\right)^{m-1}.$$
Then
\[\begin{split} H'(y)&=2 m r \int_{-\pi/2}^{\pi/2}|\sin (x+y)|^p S(x,y)\cos x dx
 \\&2 m r \int_0^{\pi/2} (|\sin (x+y)|^p-|\sin (y-x)|^p) S(x,y)\cos x dx.\end{split}\]
Let $$G(x,y)=(|\sin(y+x)|^p-|\sin(y-x)|^p)S(x,y).$$
Then for $x\in[0,\pi/2]$ and $m>1$, $$G(x,y)>0, \ \ \  \text{for}\ \ \   0<y<\pi/2$$ and $$G(x,y)<0, \ \ \  \text{for}\ \ \   \pi/2<y<\pi.$$
Further for $x\in[0,\pi/2]$ and $m<1$, $$G(x,y)<0, \ \ \  \text{for}\ \ \   0<y<\pi/2$$ and $$G(x,y)>0, \ \ \  \text{for}\ \ \   \pi/2<y<\pi.$$

The conclusion is that $\pi/2$ is the maximum of the function for $m<1$ and $0$ is its maximum for $m>1$. The case $m=1$ is trivial and in this case, the function $H$ is constant.
\end{proof}

To continue, we use the following  lemma
\begin{lemma}\label{bele}
Let $m>-1$ and $k\ge 0$ and define $$B(r,x) := \int_{-\pi}^\pi |\cos (b-x)|^k{\left(1+r^2+2 r \cos b\right)^{m}} db.$$
Then
$$B(r,x)\le \left\{
             \begin{array}{ll}
               B(1,0), & \hbox{if $m>1$;} \\
               B(1,\pi/2), & \hbox{if $m\le 1$.}
             \end{array}
           \right.$$

\end{lemma}

\begin{proof}[Proof of Lemma~\ref{bele}]
From a very similar proof as in \cite[Lemma~2.2]{caot} we get that there is $x'$ so that $B(r,x)\le B(1,x')$. Further from Lemma~\ref{marte}, we conclude that $$B(1,x')\le \left\{
             \begin{array}{ll}
               B(1,0), & \hbox{if $m>1$;} \\
               B(1,\pi/2), & \hbox{if $m\le 1$.}
             \end{array}
           \right.$$
This finishes the proof.
\end{proof}

To finish the proof of Theorem~\ref{due}, let us estimate the term
$$E_{\alpha,p}(r):=\max_\eta \int_0^{2\pi}\frac{|\alpha r \cos\eta-(2+\alpha) \cos(b+\eta)|^q }{\left(1+r^2+2 r \cos b\right)^{1-q-\frac{\alpha q}{2}}}dt$$ in \eqref{weget2}. The  we have the inequality
$$|\alpha r \cos\eta-(2+\alpha) \cos(b+\eta)|^q \le (2+\alpha)^q |\cos(b+\eta)|^q + q(2(1+\alpha))^{q-1} \alpha r.$$
The proof follows from the following inequality $$f(y)-f(0)\le \max_{0\le t\le y}|f'(t)| y$$ for $f(x) = (x+|(2+\alpha) \cos(b+\eta)|)^q$ where $y=|a r \cos \eta|$.  Let  $$P(\alpha,r)=q(2(1+\alpha))^{q-1} \alpha r$$ and $$Q(\alpha)=(2+\alpha)^q.$$ Now we have  \[\begin{split}E_{\alpha,p}(r)&\le P(\alpha,r) \int_0^{2\pi}{\left(1+r^2+2 r \cos b\right)^{q+\frac{\alpha q}{2}-1}}dt
\\&+ Q(\alpha)  \int_0^{2\pi}\frac{|\cos(b+\eta)|^q }{\left(1+r^2+2 r \cos b\right)^{1-q-\frac{\alpha q}{2}}}dt\end{split}\]
We estimate the first term. By using Lemma~\ref{bele} and get $$\int_0^{2\pi}{\left(1+r^2+2 r \cos b\right)^{q+\frac{\alpha q}{2}-1}}dt\le U_p:=\frac{2^{-1+(2+\alpha) q} \sqrt{\pi } \Gamma\left[-\frac{1}{2}+q+\frac{\alpha q}{2}\right]}{\Gamma\left[q+\frac{a q}{2}\right]}.$$ Now for $m=q+\frac{\alpha q}{2}-1$, and  $$K(\eta):=\int_0^{2\pi}\frac{|\cos(b+\eta)|^q }{\left(1+r^2+2 r \cos b\right)^{1-q-\frac{\alpha q}{2}}}dt,$$ from Lemma~\ref{marte}, we have $$V_{p,\alpha}(r):=\max_\eta{K(\eta)}=\left\{
                                                                                                                     \begin{array}{ll}
                                                                                                                       K(\pi/2), & \hbox{$m\le 1$;} \\
                                                                                                                       K(0), & \hbox{$m\ge 1$.}
                                                                                                                     \end{array}
                                                                                                                   \right.$$
Then for  \begin{equation}\label{Cap}C_{p,\alpha}(r)=\frac{c_\alpha}{(2\pi)^{1-1/p}}\left(P(\alpha,r)U_p+ Q(\alpha) V_{p,\alpha}(r)\right)\end{equation} and \begin{equation}\label{cap}c_{p,\alpha}=\frac{c_\alpha}{(2\pi)^{1-1/p}}\left(P(\alpha,1)U_p+ Q(\alpha) V_{p,\alpha}(1)\right),\end{equation} we get the desired inequalities in Theorem~\ref{due}. Observe that, for $\alpha=0$, the function $B_{p,\alpha}(r)$ and the constant $b_{p,\alpha}$ coincide with the corresponding sharp function and a constant in \cite{caot} (after normalization) where is treated the standard harmonic case. This is why our inequalities are asymptotically sharp.

\end{proof}
\section{A remark to the sharp form of Theorem~\ref{due}}

\begin{remark}
We expect that for $a>0$ the following inequality is true.  Assume  that $f$ is $\alpha-$ harmonic with its boundary function $f^\ast \in L^\infty(\mathbb{T})$. Then $$|Df(z)| \le \frac{\Gamma(1+\alpha/2)^2}{\Gamma(1+\alpha)}\frac{(1+r)^{2+\alpha}-(1-r)^{2+\alpha}}{\pi  r \left(1-r^2\right)}\|f^\ast\|_\infty.$$
\end{remark}

We confirm this conjecture for $\alpha=2$ and $\alpha = 4$. It is equivalent to the statement that the following function $\Phi$  attains its maximum for $t=\pi/2$.
Let $$\Phi(t)=\int_{-\pi}^\pi |(2+\alpha) \cos(b+t)-\alpha r \cos (t)|\frac{{\left(1+r^2+2 r \cos b\right)^{\frac{\alpha}{2}}}}{1-r^2} db,$$ and assume without lossing of generality that $t\in[0,\pi/2]$.
Then the roots of subintegral expression are
$$b_1=-t-\arccos\left[\frac{\alpha r \cos t}{2+\alpha}\right],$$ $$b_2=-t+\arccos\left[\frac{\alpha r \cos t}{2+\alpha}\right].$$
Let $$\phi(u)=\int_0^u  \partial_t \left((2+\alpha) \cos(b+t)-\alpha r \cos (t)\right)\frac{{\left(1+r^2+2 r \cos b\right)^{\frac{\alpha}{2}}}}{1-r^2} db.$$
Let $$I=I(t,b)=((2+\alpha) \cos(b+t)-\alpha r \cos (t))\frac{{\left(1+r^2+2 r \cos b\right)^{\frac{\alpha}{2}}}}{1-r^2}.$$ Then
\[\begin{split}
\Phi(t)&=-\int_{-\pi}^{b_1} I db+\int_{b_1}^{b_2} I db-\int_{b_2}^{\pi} I db
\\&= 2\int_{b_1}^{b_2} I db-\int_{-\pi}^{\pi} I db.\end{split}\]
Let $$\phi(u)=\int_0^u  \partial_t \left((2+\alpha) \cos(b+t)-\alpha r \cos (t)\right)\frac{{\left(1+r^2+2 r \cos b\right)^{\frac{\alpha}{2}}}}{1-r^2} db.$$
Then $$\phi(u)=-\int_0^u  \left((2+\alpha) \sin(b+t)-\alpha r \sin (t)\right)\frac{{\left(1+r^2+2 r \cos b\right)^{\frac{\alpha}{2}}}}{1-r^2} db.$$
Furthermore $$\Phi'(t)=2\phi(b_2)-2\phi(b_1)+\phi(-\pi)-\phi(\pi).$$
Next $$\Phi'(t)=2\int_{b_2}^{b_1} (\alpha r \sin (t)-(2+\alpha) \sin(b+t))\frac{{\left(1+r^2+2 r \cos b\right)^{\frac{\alpha}{2}}}}{1-r^2} db+f(-\pi)-f(\pi).$$

For $\alpha=2$ we have $$\Phi'(t)=\frac{4 r \left(2 \left(1-r^2\right) \arcsin\left[\frac{1}{2} r \cos t\right]+3 r \cos t \sqrt{4-r^2 \cos t^2}\right) \sin t}{1-r^2}$$ and this function is positive for $t \in [0,\pi/2]$ and $r\in[0,1]$.
So the maximum is in $t = \pi/2$.

For $\alpha=4$ we get  \[\begin{split}\Phi'(t)&=\frac{16 r}{{27 \left(1-r^2\right)}} \times \Bigg (27 \left(2-r^2-r^4\right) \arcsin\left[\frac{2}{3} r \cos t \right] \sin t \\&+r \sqrt{9-4 r^2 \cos^2 t } \left(9+31 r^2+10 r^2 \cos(2t)\right) \sin(2t)\Bigg),\end{split}\] which is positive.

So we proved the following proposition

\begin{theorem}
Assume  that $f$ is $2-$ harmonic with its boundary function $f^\ast \in L^\infty(\mathbb{T})$. Then $$|f(z)|\le \frac{4 \left(1+r^2\right)}{\pi  \left(1-r^2\right)}\|f^\ast\|_\infty.$$

For $4-$harmonic functions we have the following sharp inequality $$ |f(z)|\le \frac{6+20 r^2+6 r^4}{3 \pi  \left(1-r^2\right)}\|f^\ast\|_\infty.$$
\end{theorem}

\begin{proof}[Proof of Theorem~\ref{three}]
Let $U_\alpha$ be an $\alpha-$harmonic function so that for $|\zeta|=1$, $U_\alpha(\zeta)=1$ for $\Im \zeta>0$ and $U_\alpha(\zeta)=-1$ for $\Im \zeta<0$. Let $u$ be a real $\alpha-$harmonic so that $|u(z)|<1$ and that $u(0)=0$ and let $|z|<1$. We claim that $|u(z)|\le |U(i|z|)|$. Since a composition of $\alpha-$harmonic functions and a rotation is $\alpha-$harmonic, we can assume without losing the generality that $z=ir$.
First of all \begin{equation}\label{weneed}\begin{split}U(ir ) &= \frac{c_\alpha}{2\pi} \bigg(\int_0^{\pi}\frac{(1-r^2)^{\alpha+1}}{\left(1+r^2-2 r \cos (t-\pi/2)\right)^{\alpha/2+1}}dt \\&- \int_{\pi}^{2\pi}\frac{(1-r^2)^{\alpha+1}}{\left(1+r^2-2 r \cos (t-\pi/2)\right)^{\alpha/2+1}}dt\bigg).\end{split}\end{equation}
We need to prove that $$\frac{c_\alpha}{2\pi} \int_{\mathbb{T}} \frac{(1-r^2)^{\alpha+1}}{\left(1+r^2-2 r \cos (t-\pi/2)\right)^{\alpha/2+1}} u(e^{it})dt\le  U(ir).$$
which is equivalent with the following inequality
$$\int_{\mathbb{T}^-} \frac{(1+u(e^{it}))}{\left(1+r^2-2 r \sin t\right)^{\alpha/2+1}} dt\le  \int_{\mathbb{T}^+} \frac{ (1-u(e^{it}))}{\left(1+r^2-2 r \sin t\right)^{\alpha/2+1}}dt.$$

Since $u(0)=0$, $$\int_{\mathbb{T}^-} u(e^{it}) dt=-\int_{\mathbb{T}^+} u(e^{it}) dt$$

Thus

$$\int_{\mathbb{T}^-} \frac{(1+u(e^{it}))}{\left(1+r^2-2 r \sin t\right)^{\alpha/2+1}} dt\le \int_{\mathbb{T}^-} \frac{(1+u(e^{it}))}{\left(1+r^2 \right)^{\alpha/2+1}} dt$$

$$= \int_{\mathbb{T}^+} \frac{(1-u(e^{it}))}{\left(1+r^2 \right)^{\alpha/2+1}} dt\le  \int_{\mathbb{T}^+} \frac{(1-u(e^{it}))}{\left(1+r^2 -2r \sin t\right)^{\alpha/2+1}} dt .$$

Thus we proved $u(z)\le U_\alpha(z)$. From \eqref{weneed} we obtain  \begin{equation}\label{weneed1}\begin{split}U(ir ) &= \frac{c_\alpha}{2\pi} \bigg(\int_0^{\pi}\frac{(1-r^2)^{\alpha+1}}{\left(1+r^2-2 r \sin t \right)^{\alpha/2+1}}dt \\&- \int_{\pi}^{2\pi}\frac{(1-r^2)^{\alpha+1}}{\left(1+r^2-2 r \sin t\right)^{\alpha/2+1}}dt\bigg).\end{split}\end{equation}

Now $$\frac{(1-r^2)^{\alpha+1}}{\left(1+r^2-2 r \sin t\right)^{\alpha/2+1}}=\sum_{k=0}^\infty a_k \sin^k t $$ where
$$a_k=\frac{(-2k)^k \left(1-r^2\right)^{1+\alpha}}{ \left(1+r^2\right)^{1+\frac{\alpha}{2}+k}} \binom{-1-\frac{\alpha}{2}}{k}.$$
Further $$\int_0^{\pi} \sin^k t dt = \frac{\sqrt{\pi } \Gamma\left[\frac{1+k}{2}\right]}{\Gamma\left[1+\frac{k}{2}\right]}$$ and

$$\int_\pi^{2\pi} \sin^k t dt = \frac{(-1)^k \sqrt{\pi } \Gamma\left[\frac{1+k}{2}\right]}{\Gamma\left[1+\frac{k}{2}\right]}.$$

By summing we get $$U(ir)=\frac{c_\alpha}{2\pi}\frac{4 (2+\alpha) (1-r^2)^{1+\alpha} r}{  \left(1+r^2\right)^{2+\frac{\alpha}{2}}}F\left[\begin{array}{ccc}
                                            1 \ \ \ 1+\frac{\alpha}{4} \ \ \ \frac{3}{2}+\frac{\alpha}{4} \\
                                             \frac{3}{2} \ \ \  \frac{3}{2}
                                          \end{array}; \frac{4 r^2}{\left(1+r^2\right)^2}\right].$$
\end{proof}

\section*{Ethics declarations}
\subsection*{Conflict of interest}
The author declares that he has not conflict of interest.

\subsection*{Data statement}
Data sharing not applicable to this article as no datasets were generated or analysed during the current study.

\subsection*{Acknowledgments} I would like to thank  A. Olofsson for drawing my attention to the paper \cite{ahern}.

\end{document}